\documentclass[11pt,oneside]{amsart}
\usepackage[utf8]{inputenc}

\usepackage[english]{babel}

\usepackage[dvips]{graphics,epsfig}
\usepackage{times}
\usepackage{amsmath,amsthm,amssymb}
\usepackage{array}
\usepackage{bbm}
\usepackage{xcolor}
\usepackage{verbatim}

\usepackage{geometry}

\def\AAA{\mathbb{A}}

\def\MM{\mathcal{M}}

\def\NNN{\mathbb{N}}

\def\QQQ{\mathbb{Q}}

\def\UU{\mathcal{U}}

\def\a{\alpha}
\def\b{\beta}
\def\g{\gamma}
\def\d{\delta}

\def\D{\Delta}

\def\nmdeg{\operatorname{nmdeg}}

\DeclareMathOperator{\tp}{tp}
\DeclareMathOperator{\acl}{acl}
\DeclareMathOperator{\dcl}{dcl}

\def\trdeg{\operatorname{trdeg}}
\def\ord{\operatorname{ord}}

\newtheorem{theorem}{Theorem}[section]
\newtheorem{thm}[theorem]{Theorem}
\newtheorem{lemma}[theorem]{Lemma}
\newtheorem{cor}[theorem]{Corollary}
\newtheorem{prop}[theorem]{Proposition}
\newtheorem{fact}[theorem]{Fact}

\newtheorem{cond}[theorem]{Condition}

\theoremstyle{definition}
\newtheorem{definition}[theorem]{Definition}

\theoremstyle{remark}

\newcommand{\m}{\mathbb }

\def \D {\Delta}

\newcommand{\gen}[1]{\left\langle#1\right\rangle}

\def\Ind{\setbox0=\hbox{$x$}\kern\wd0\hbox to 0pt{\hss$\mid$\hss} \lower.9\ht0\hbox to 0pt{\hss$\smile$\hss}\kern\wd0}
\def\Notind{\setbox0=\hbox{$x$}\kern\wd0\hbox to 0pt{\mathchardef \nn=12854\hss$\nn$\kern1.4\wd0\hss}\hbox to 0pt{\hss$\mid$\hss}\lower.9\ht0 \hbox to 0pt{\hss$\smile$\hss}\kern\wd0}

\numberwithin{equation}{section}

\begin{document}

\title{Strong minimality of triangle functions}
\author[G. Casale]{Guy Casale}
\author[M. DeVilbiss]{Matthew DeVilbiss}
\author[J. Freitag]{James Freitag}
\author[J. Joel Nagloo]{Joel Nagloo}
\address{Guy Casale, Univ Rennes, CNRS, IRMAR-UMR 6625, F-35000 Rennes, France}
\email{guy.casale@univ-rennes1.fr}
\address{Mathew DeVilbiss, University of Illinois Chicago, Department of Mathematics, Statistics,
and Computer Science, 851 S. Morgan Street, Chicago, IL, USA, 60607-7045.}
\email{mdevil2@uic.edu}
\address{James Freitag, University of Illinois Chicago, Department of Mathematics, Statistics,
and Computer Science, 851 S. Morgan Street, Chicago, IL, USA, 60607-7045.}
\email{jfreitag@uic.edu}
\address{Joel Nagloo, University of Illinois Chicago, Department of Mathematics, Statistics,
and Computer Science, 851 S. Morgan Street, Chicago, IL, USA, 60607-7045.}
\email{jnagloo@uic.edu}
\thanks{G. Casale is partially funded by CAPES-COFECUB project MA 932/19. J. Freitag is partially supported by NSF CAREER award 1945251. J. Nagloo is partially supported by NSF grant DMS-2203508.}

\maketitle
\begin{abstract}
In this manuscript we give a new proof of strong minimality of certain automorphic functions, originally results of Freitag and Scanlon (2017), Casale, Freitag and Nagloo (2020), Bl{\'a}zquez-Sanz, Casale, Freitag and Nagloo (2020). Our proof is shorter and conceptually different than those presently in the literature.
\end{abstract}

\section{Introduction}
In \cite{mahler1969algebraic}, Mahler proved that $j \left( \frac{ \log q}{\pi i } \right)$ satisfies no first or second order differential equation over $\m C (q)$, while conjecturing that any automorphic function of a Fuchsian subgroup of $SL_2 $ with at least $3$ limit points with respect to the action on the upper half plane should also have this property.\footnote{Any of these automorphic functions does satisfy a differential equation of third order over $\m C(q)$.} Mahler established the result under additional restrictions on the Fuchsian group $G \leq SL_2$, and Nishioka \cite{Nishioka} proved Mahler's conjecture, showing: 

\begin{thm} \label{Nishi1}
\cite{Nishioka} Let $G$ be a discontinuous subgroup of $SL_2$ which possesses at least three limit points. Let $u$ be a nonzero complex number. Then any automorphic function of $G$ satisfies no algebraic differential equation of order two over $\m C (z, e^{uz} ).$ 
\end{thm}

Given a hyperbolic triangle in the upper half plane with angles $0$ or $\pi /n $ for $n \in \m N_{>1}$, through successive hyperbolic reflections over the sides of the triangle, one can tile the upper half plane. The group of transformations generated by the reflections is an example of a (hyperbolic) triangle group; such groups are discrete Zariski-dense subgroups of $SL_2$ and their collection of limit points is $\m R \cup \{ \infty \}$. They are thus are a special case of the groups considered in Theorem \ref{Nishi1}. Schwarz triangle functions are automorphic functions for these hyperbolic triangle groups, and Nishioka proved a strengthening of Theorem \ref{Nishi1} for Schwarz triangle functions: 

\begin{thm} \label{nishi2}
Schwarz triangle functions are not $2$-reducible\footnote{For the definition of $n$-reducibility, see Subsection \ref{nofibs}.} over $\m C$.
\end{thm}

Nishioka's above theorems are special cases of a conjecture of Painlev\'e \cite[pp 444-445 and p 519]{painleve1leccons}
proved in \cite{casale2020ax}, in which the field $\m C (z, e^{uz} )$ can be replaced with an arbitrary differential field extension. This generalization is equivalent to the \emph{strong minimality} of the equation, a notion having its origins in model theory. In this short article we give a new proof of the strong minimality of Schwarz triangle functions (originally results of \cite{freitag2017strong, blazquez2020some}, later \cite{aslanyan2021ax}) using the recent work of \cite{freitagmoosa}. 


Our proof utilizes the connection between solutions of a Schwarzian equation and a certain associated linear differential equation, via Puiseux series. 
From there, we use Nishioka's theorem \ref{nishi2} with a model theoretic study of fibrations inspired by the work of \cite{moosa2014model} to prove the strong minimality of Schwarz triangle functions. The motivation for establishing the strong minimality of certain differential equations has been discussed many places and ranges from functional transcendence theorems \cite{casale2020ax} to diophantine results \cite{freitag2017strong, HrML, HPnfcp} to understanding solutions of equations from special classes of functions \cite{freitag2021not, freitag2022equations}.

\subsection{Organization of this paper} 
In Section \ref{thenonminsection}, we give a brief explanation of the work of Freitag and Moosa \cite{freitagmoosa} and its specific consequences in our context. In Section \ref{ord2}, we develop a connection between order two subvarieties of our nonlinear equations and algebraic solutions of certain associated Riccati equations. 
In section \ref{nishgen} we finish the proof of the strong minimality of Schwarz triangle functions using Section \ref{ord2}.

\section{Nonminimality} \label{thenonminsection}
\subsection{The degree of nonminimality} 
It is a well-known consequence of stable embeddedness that the canonical base of a forking extension of a stable type can be found in the algebraic closure of an initial segment of a Morley sequence in the type. In very recent work, Freitag and Moosa \cite{freitagmoosa} introduced the following definition which strengthens the bound on the length of a sequence one needs to consider. 

\begin{definition}
Suppose $p\in S(A)$ is a stationary type of $U$-rank larger than one. 
By the {\em degree of nonminimality} of $p$, $\nmdeg(p)$, we mean the least positive integer $k$ such that there is a sequence of realisations of $p$ of length $k$, say $(a_1,\dots,a_k)$, and $p$ has a nonalgebraic forking extension over $a_1,\dots,a_k$.
\end{definition}

In the theory of differentially closed fields of characteristic zero, Freitag and Moosa \cite{freitagmoosa} give an upper bound for the degree of nonminimality in terms of the U-rank of the type. 

\begin{thm} \label{boundMorley}
Let $p \in S(k)$ have finite rank. Then $\nmdeg(p) \leq RU(p)+1.$
\end{thm}

In \cite{devilbiss2021generic}, DeVilbiss and Freitag use the bounds of \cite{freitagmoosa} along with computations involving the Lascar rank of underdetermined systems of differential equations to give a proof of the strong minimality of generic differential equations of sufficiently high degree. One of the challenging aspects of generalizing the proofs of \cite{devilbiss2021generic} to various other classes of equations is the complexity of the series of algebraic reductions used to calculate the rank of a certain associated linear system. 

In this manuscript, we will use the following result of \cite[a slight restatement of Proposition 3.3]{freitagmoosa} which gives a stronger bound than Theorem \ref{boundMorley} in general for types in totally transcendental theories, but only with additional hypotheses on the type:

\begin{prop} \label{nmbound} 
Suppose that $p \in S(A)$ is of finite $U$-rank. If $p$ is \emph{not} almost internal to a non locally modular minimal type, then $\nmdeg{p} \leq 1$. 
\end{prop}

In the theory of differentially closed fields, a strong form of the Zilber trichotomy holds - the non locally modular minimal type can be without loss of generality assumed to be the type of a generic constant. By results of \cite{freitag2017finiteness}, a differential variety has infinitely many co-order one subvarieties over some finitely generated differential field $k$ \emph{if and only if} $X$ is nonorthogonal to the constants. So, it follows that:

\begin{prop} \label{nmbound1}
Suppose that $p \in S(A)$ is the generic type of a finite rank affine differential variety $X$. If $X$ has no co-order one differential subvarieties, then $\nmdeg{p} \leq 1$. 
\end{prop}

\subsection{No proper fibrations} \label{nofibs}
The following definition comes from \cite{moosa2014some} and applies generally in the case that $T$ is a complete theory admitting elimination of imaginaries and $p$ is a type of finite $U$-rank. 
\begin{definition}
A stationary type $p = \tp (a/A)$ \emph{admits no proper fibrations} if whenever $c \in \dcl (Aa) \setminus \acl (A)$, we must have $a \in \acl (Ac)$. 
\end{definition}

The following result \cite[Proposition 2.3]{moosa2014some} gives a useful restriction on those types with no proper fibrations. 
\begin{prop} \label{nofib}
Suppose that $p = \tp (a/A )$ is stationary and admits no proper fibrations. Then either
\begin{enumerate}
    \item $p$ is almost internal to a non locally modular minimal type, or 
    \item $a$ is interalgebraic over $A$ with a finite tuple of independent realizations of a locally modular minimal type over $A$. 
\end{enumerate}
\end{prop}

The following condition was developed in a series of papers of Nishioka \cite{nishioka1990painleve, nishiokaII}:

\begin{definition}
Let $y$ be differentially algebraic over a differential field $k$. We say $a$ is \emph{$r$-reducible over $k$} if there exists a finite chain of differential field extensions, $$k=R_0 \subset R_1 \subset \ldots R_m$$ such that $a \in R_m$ and $\trdeg{R_i/R_{i-1}} \leq r.$ 
\end{definition}

\begin{lemma} \label{nofib1}
Suppose $p=\tp (a/k ) $ is the generic type of an order $3$ differential equation over $k$. If $a$ is not $2$-reducible over $k$, then $p=\tp (a/k ) $ has no proper fibrations. 
\end{lemma}

\begin{proof}
Suppose that $\tp (a/k )$ has a proper fibration. Let $b \in \dcl (a/k)$ be such that $a \notin \acl (b/k).$ Then note that $3= \trdeg{k \langle a \rangle /k } =\trdeg{k \langle a, b \rangle /k  },$ but since $a \notin \acl (b/k)$ and $b \notin \acl (k)$, we must have $1 \leq \trdeg{k  \langle b \rangle /k } \leq 2$. It follows that $k \subset k \langle b \rangle \subset k \langle a \rangle$ implies that $a$ is $2$-reducible a contradiction.  
\end{proof}

As mentioned in the introduction, in \cite{nishiokaII}, Nishioka proved that Schwarz triangle functions are not $2$-reducible over $\m C$. Though most of the above conditions deal only with the generic solutions of the differential equations we consider, a beautiful elementary argument of Nishioka \cite{Nishioka} shows that the differential equations satisfied by Fuchsian automorphic functions have no proper infinite differential subvarieties over $\m C$, so any nonalgebraic solution is generic over $\m C$. Thus, in order to establish that these equations are strongly minimal, one needs only to verify that their generic types have no nonalgebraic forking extensions.

\section{No order two subvarieties} \label{ord2}

We now aim to show that modulo a certain condition, the Schwarzian equation $(\star)$ has no order two subvarieties. Consider the equation:
\begin{equation}\tag{$\star'$}
S_{\frac{d}{dt}}(y) +(y')^2\cdot R(y) =0\label{(*)},
\end{equation}
where $\frac{dy}{dt} = y'$, $S_{\frac{d}{dt}}(y) = \left( \frac{y''}{y'}\right)' -\frac{1}{2} \left( \frac{y''}{y'}\right)^2$ and $R$ is rational over $\mathbb{C}$. To Equation (\ref{(*)}), we have an associated Riccati equation:
\begin{equation}\tag{$\star\star$}
\frac{du}{dy}+u^2+\frac{1}{2}R(y) =0.
\end{equation}

\begin{cond}\label{Ric}
The Riccati equation ($\star\star$) has no solution in $\mathbb{C}(y)^{alg}$.
\end{cond}

\begin{lemma} \label{juststrmin}
Let $(k,\partial)$ be any differential field extension of $\mathbb{C}$ and let us assume that Condition \ref{Ric} holds. If $y$ is a solution of the Schwarzian equation $(\star')$, then
\[\trdeg k\gen{y}/k\neq 2.\] In other words, Equation $(\star')$ has no order two subvarieties.
\end{lemma}

\begin{proof}
Suppose $\trdeg k\gen{y}/k=2$. 
\par Let $L=k(y)^{alg}$ and $\partial$ be the extension of the derivation of $k$ such that $\partial(y) = 0$. Since $y'$ is transcendental over $L$ we can work in the field $L\gen{\gen{1/y'}}$ of Puiseux series in $1/y'$ over $L$. We have a well defined derivation
\[(\sum a_iy'^{\lambda_i})'=\sum \partial(a_i) y'^{\lambda_i}+\sum \frac{\partial a_{i}}{\partial y }y'^{\lambda_i+1}+\sum \lambda_ia_iy'^{\lambda_i-1}y''\] where $\lambda_i$, $i\geq0$, are descending rational numbers with a common denominator, $a_i\in L$ and $a_0\neq0$. It follows that $K\gen{y}^{alg}$ is a differential subfield of $L\gen{\gen{1/y'}}$. Now let $u=\frac{y''}{y'^2}$ and by replacing in equation $(\star')$ we get
\[\frac{u'}{y'}+\frac{1}{2}u^2+R(y)=0.\]
Writing $u=\sum a_iy'^{\lambda_i}$ with $i\geq0$ and $a_0\neq0$ and differentiating
\[\frac{u'}{y'}=\sum \partial(a_i)y'^{\lambda_i-1}+\sum \frac{ \partial a_{i}}{\partial y}y'^{\lambda_i}+\sum \lambda_ia_iy'^{\lambda_i}\sum a_i y'^{\lambda_i} = -\frac{1}{2} \left(\sum a_iy'^{\lambda_i}\right)^2  -R(y).\]
Since $R(y)$ (a coefficient of $y'^0$) appear non-trivially in the above and since $R(y)\not\in k^{alg}$, it is easily seen that $\lambda_0=0$ and 
 \[\frac{\partial a_{0}}{\partial y}+\frac{1}{2}a_0^2+R(y) =0.\]
 But then as before $a_0/2$ can be seen as an algebraic solution of $(\star\star)$ in $\mathbb{C}(y)^{alg}$ contradicting Condition \ref{Ric}.\\
 \end{proof}
 
Let us say a few words about when Condition \ref{Ric} holds. First, it is well-known that the Riccati equation ($\star\star$) has an algebraic solution if and only if a certain associated order two linear equation has Liouvillian solutions (see \cite[1.3]{KOVACIC19863}). In \cite{KOVACIC19863}, Kovacic gives an algorithm to determine when an order two linear equation has Liouvillian solutions, and so, by the above equivalence, Kovacic's algorithm can be used to determine when Condition \ref{Ric} holds. 
Moreover,  when the rational function $R(y)$ is of ``triangular form''


\[R(y) = R_{\alpha,\beta,\gamma}(y)=\frac{1}{2}\left(\frac{1-{\beta}^{-2}}{y^2}+\frac{1-{\gamma}^{-2}}{(y-1)^2}+\frac{{\beta}^{-2}+{\gamma}^{-2}-{\alpha}^{-2}-1}{y(y-1)}\right)\] with $\alpha,\beta,\gamma\in \mathbb{C}\cup \{ \infty \}$, one can give precise conditions under which Condition \ref{Ric} holds. The particular rational function $R_{\alpha,\beta,\gamma}(y)$ given above is referred to as being in triangular form because a Schwarz triangle function with angles of the associated hyperbolic triangle given by $(\frac{\pi}{\alpha},\frac{ \pi }{\beta} , \frac{\pi }{\gamma}),$ for integers $\alpha, \beta, \gamma\in\m{N} \cup \{\infty \}$ satisfies $S_{\frac{d}{dt}}(y) +(y')^2\cdot R_{\alpha,\beta,\gamma}(y) =0$. This classification of algebraic solutions to \ref{Ric} applies more generally to the case where $\alpha, \beta , \gamma \in \m C \cup \{ \infty \}$ and comes ultimately from Theorem I of \cite{Kimura}, but is explained in detail in \cite[Proposition 4.4]{blazquez2020some}. 

\begin{fact} \label{kimurafact} Condition \ref{Ric} holds of $$\frac{du}{dy}+u^2+\frac{1}{2}R_{\alpha, \beta, \gamma }(y) =0$$ as long as none of the following conditions hold.
\begin{enumerate}
    \item The quantities $\alpha^{-1}$ or $-\alpha^{-1}$, $\beta^{-1}$ or $-\beta^{-1}$ and $\gamma^{-1}$ or $-\gamma^{-1}$ take, in an arbitrary order, values given in the following table:

\begin{center}
\renewcommand{\arraystretch}{1.2}
\begin{tabular}{|c|c|c|c|c|}
\hline 
  & $\pm \alpha^{-1}$ & $\pm\beta^{-1}$ & $\pm \gamma^{-1}$ & \\ 
\hline 
1  & $\frac{1}{2}+\ell$ & $\frac{1}{2}+m$ & arbitrary & \\ 
\hline 
2 & $\frac{1}{2} + \ell$ & $\frac{1}{2}+m$ & $\frac{1}{2}+ n$ & \\  
\hline
3 & $\frac{2}{3}+\ell$ & $\frac{1}{3}+m$ & $\frac{1}{4}+n$ & $\ell+m+n$ even \\ 
\hline
4 &  $\frac{1}{2}+\ell$ & $\frac{1}{3}+m$ & $\frac{1}{4}+n$ &  \\ 
\hline
5 &  $\frac{2}{3}+\ell$ & $\frac{1}{4}+m$ & $\frac{1}{4}+n$ & $\ell+m+n$ even  \\
\hline
6 &  $\frac{1}{2}+\ell$ & $\frac{1}{3}+m$ & $\frac{1}{5}+n$ &   \\ 
\hline
7 &  $\frac{2}{5}+\ell$ & $\frac{1}{3}+m$ & $\frac{1}{3}+n$ & $\ell+m+n$ even  \\ 
\hline
8 &  $\frac{2}{3}+\ell$ & $\frac{1}{5}+m$ & $\frac{1}{5}+n$ & $\ell+m+n$ even  \\ 
\hline
9 &  $\frac{1}{2}+\ell$ & $\frac{2}{5}+m$ & $\frac{1}{5}+n$ & $\ell+m+n$ even  \\ 
\hline
10 &  $\frac{3}{5}+\ell$ & $\frac{1}{3}+m$ & $\frac{1}{5}+n$ & $\ell+m+n$ even  \\ 
\hline
11 &  $\frac{2}{5}+\ell$ & $\frac{2}{5}+m$ & $\frac{2}{5}+n$ & $\ell+m+n$ even  \\ 
\hline
12 &  $\frac{2}{3}+\ell$ & $\frac{1}{3}+m$ & $\frac{1}{5}+n$ & $\ell+m+n$ even  \\ 
\hline
13 &  $\frac{4}{5}+\ell$ & $\frac{1}{5}+m$ & $\frac{1}{5}+n$ & $\ell+m+n$ even  \\ 
\hline
14 &  $\frac{1}{2}+\ell$ & $\frac{2}{5}+m$ & $\frac{1}{3}+n$ & $\ell+m+n$ even  \\ 
\hline
15 &  $\frac{3}{5}+\ell$ & $\frac{2}{5}+m$ & $\frac{1}{3}+n$ & $\ell+m+n$ even  \\ 
\hline
\end{tabular}
\end{center}
where $\ell$, $m$, $n$ stand for arbitrary integers.

\item At least one of the four complex numbers,
    $\alpha^{-1} + \beta^{-1} + \gamma^{-1}$, 
    $-\alpha^{-1} + \beta^{-1} + \gamma^{-1}$,
    $\alpha^{-1} - \beta^{-1} + \gamma^{-1}$,
    $\alpha^{-1} + \beta^{-1} - \gamma^{-1}$ is an odd integer.
\end{enumerate}

\end{fact}

\begin{prop} \label{theorder2result}
The differential equations satisfied by Schwarz triangle functions have no order two subvarieties. 

\end{prop}

\begin{proof}
Condition (2) of Fact \ref{kimurafact} can never hold for $\alpha, \beta, \gamma$ coming from a hyperbolic triangle with angles $(\frac{\pi}{\alpha},\frac{ \pi }{\beta} , \frac{\pi }{\gamma}),$ as hyperbolicity of the triangle implies that $\frac{1}{\alpha } + \frac{1}{\beta } + \frac{1}{\gamma} <1$. 

With $\alpha, \beta $ integers and $\frac{1}{\alpha } + \frac{1}{\beta } + \frac{1}{\gamma} <1$, we can never have rows 1,2 of Condition (1), since the conditions on $\alpha$ and $\beta $ would require both to be $2$, violating hyperbolicity. Rows 3, 5, 7, 8, 10, 11, 12, 13, and 15 can never occur for an integer $\alpha$. For integers $\alpha, \beta, \gamma$, row 4 would require $\alpha =2, \beta =3, \gamma = 4$, violating hyperbolicity. Similarly, we would have a violation of hyperbolicity if row 6 holds. The conditions on $\beta $ in rows 9 and 14 can not hold for an integer $\beta$. But we've now ruled out conditions (1) and (2) of Fact \ref{kimurafact}. 

Thus, Condition \ref{Ric} holds when $R (y) = R_{\alpha,\beta,\gamma}(y)$ with $\alpha, \beta, \gamma$ coming from a hyperbolic triangle with angles $(\frac{\pi}{\alpha},\frac{ \pi }{\beta} , \frac{\pi }{\gamma})$. Now, by Lemma \ref{juststrmin}, the equation $S_{\frac{d}{dt}}(y) +(y')^2\cdot R(y) =0$ satisfied by the Schwarz triangle equation has no order two subvarieties. 
\end{proof}

\section{Strong minimality in the style of Nishioka} \label{nishgen} 
In this section, we use the results of Section \ref{ord2} with the model theoretic notions of Section \ref{thenonminsection} to give a new proof of the strong minimality of the equations satisfied by Schwarz triangle functions (and so including the $j$-function). In particular these results strengthen the main theorems of \cite{nishiokaII, nishioka1990painleve}, generalize the main theorem of \cite{freitag2017strong}, and give new proofs of results of \cite{blazquez2020some, casale2020ax} in the case of Schwarz triangle functions.

\begin{prop}
The differential equations satisfied by Schwarz triangle functions are strongly minimal.
\end{prop}
\begin{proof}
Denote by $X$ the solution set of such an equation. By the remarks concluding Section \ref{thenonminsection}, we need only establish the minimality of the generic type of $X$. 
By Theorem \ref{nishi2} and Lemma \ref{nofib1}, the generic type $p = \tp (a /\m C)$ of such an equation has no fibrations. So, it follows by Proposition \ref{nofib} that $p$ is either almost internal to $\m C$ or is interalgebraic over $\m C$ with the product of a locally modular type with itself. 

If $p$ is almost internal to the constants, then it follows that $p$ is nonorthogonal to $\m C$ and so there is a nonconstant definable map from $X$ to $\m C$. A generic fiber of this map gives an order $2$ subvariety of $X$, which would contradict Proposition \ref{theorder2result}. 

If $p $ is interalgebraic with $q \otimes q \ldots \otimes q$, where $q$ is locally modular, then either $q$ is an order one type and the product is $q \otimes q \otimes q$ or $q$ is order three and $p$ is minimal. In the first case, we have an algebraic correspondence: $\phi: p \rightarrow q \otimes q \otimes q $ and a natural map $\pi_1: q \otimes q \otimes q \rightarrow q$. Let $a$ be a $\m C$-generic realization of $q$. Then $\phi^{-1} (\pi_1^{-1} ( a))$ is a type of order 2 over $\m C \langle a \rangle $ and is contained in $X$. This is impossible by Proposition \ref{theorder2result}. So, it must be that $p$ is minimal. 
\end{proof}
Since $X$ is defined over $\m C$, it follows from Proposition 5.8 of \cite{casale2020ax} that $X$ is geometrically trivial. Generalizing Nishioka's result, Theorem \ref{nishi2}, to larger classes of automorphic functions, would similarly yield a proof the strong minimality of the associated equations - can one give a proof in the style of \cite{nishiokaII}? Doing so for Fuchsian functions of the first kind would allow one to recover the main strong minimality results of \cite{casale2020ax}.

\bibliography{research}{}
\bibliographystyle{plain}
\end{document}